\documentclass[11pt,reqno]{amsart}

\usepackage[T1]{fontenc}
\usepackage{graphicx}
\usepackage{cite}
\usepackage{bm}

\usepackage{color}
\definecolor{MyLinkColor}{rgb}{0,0,0.4}

\newcommand{\R}{{\mathbb R}}

\newcommand{\Z}{{\mathbb Z}}
\newcommand{\N}{{\mathbb N}}

\newcommand{\s}{\mathbb S}
\newcommand{\h}{\rho}

\newcommand{\A}{\mathcal{A}}
\newcommand{\B}{\mathcal{B}}
\newcommand{\T}{\mathcal{T}}
\newcommand{\kL}{\mathcal{L}}

\newcommand{\ov}{\overline}
\newcommand{\p}{\partial}
\newcommand{\0}{\Omega}

\newcommand{\e}{\varepsilon}
\newcommand{\wh}{\widehat}

\newcommand{\Kern}{\mathop{\rm Ker}\nolimits}
\newcommand{\Isom}{\mathop{\rm Isom}\nolimits}

\newcommand{\im}{\mathop{\rm Im}\nolimits}

\newcommand{\tr}{\mathop{\rm tr}\nolimits}
\newcommand{\Ad}{\mathop{\rm Ad}\nolimits}
\newcommand{\spn}{\mathop{\rm span}\nolimits}

\newcommand{\Mf}{\mathcal F}

\newtheorem{thm}{Theorem}[section]

\newtheorem{lemma}[thm]{Lemma}

\theoremstyle{remark} 
\newtheorem{rem}[thm]{Remark}
   \theoremstyle{remark}
          \newtheorem*{acknowledgement*}{Acknowledgement}

\numberwithin{equation}{section}
\title[On stratified capillary-gravity water waves]{On the existence of steady periodic capillary-gravity  stratified water waves}


\author[D. Henry]{David Henry}
\address{School of Mathematical Sciences, Dublin City University, Glasnevin, Dublin 9, Ireland.}
\email{david.henry@dcu.ie}

\author[B.--V. Matioc]{Bogdan--Vasile Matioc}
\address{Institut f{\"u}r Angewandte Mathematik, Leibniz Universit{\"a}t Hannover, Welfengarten~1, 30167 Hannover, Germany.}
\email{matioc@ifam.uni-hannover.de}

\usepackage[colorlinks=true,linkcolor=MyLinkColor,citecolor=MyLinkColor]{hyperref} 

\begin{document}

\begin{abstract}
We  prove the existence of small steady periodic capillary-gravity  water waves for general stratified flows, where we allow for stagnation points in the flow. We establish the existence of both laminar and non-laminar 
flow solutions for the governing equations. This is achieved by using bifurcation theory and estimates based on the ellipticity of the system, where we regard, in turn, the mass-flux and surface tension as bifurcation parameters.
\end{abstract}
 
\maketitle

\section{Introduction}\label{Intro}
In the following paper we prove the existence of small-amplitude two-dimensional steady periodic capillary-gravity stratified water waves, where we do not exclude stagnation points from the flow. Stratified water waves are heterogeneous flows where 
the density varies as a function of the streamlines.
  Physically, stratification is a very interesting phenomenon,
 and fluid density may be caused to fluctuate by a plethora of factors -- for example, salinity, temperature, pressure, topography, oxygenation.
 Mathematically, allowing for heterogeneity adds severe complications to the governing equations, making their analysis even more intractable. For capillary-gravity waves,
 the effects of surface tension are incorporated in the governing equations by adding a term proportional to the curvature of the wave profile into the surface boundary condition, which also complicates matters significantly.  
Nevertheless, as we have shown in our paper \cite{HM}, the remarkably nice regularity properties which 
have been recently proven to hold for a wide-variety of homogeneous flows \cite{CE3,Hen,Hen1,Mat,Mat1}, in the main apply
 also to these stratified flows. 

The first results concerning the existence of small-amplitude waves for stratified flows were obtained by Dubreil-Jacotin \cite{Dub2}, in $1937$. Interestingly, prior to this Dubreil-Jacotin successfully adapted the famous Gerstner's water wave \cite{Con,Hen2}, which is one of the few examples of explicit solutions for the full governing equations which exist, and which describes a rotational water wave in a homogeneous fluid, to the setting of a stratified fluid in \cite{Dub}.  More recently, rigorous results concerning the existence of small and large amplitude stratified flows were obtained via bifurcation methods in     \cite{Wa1}, building on techniques first applied to rotational homogeneous flows in \cite{CS2}. The existence of small and large amplitude waves for stratified flows with the additional complication of surface tension was then addressed in \cite{Wa3}, building on from local existence results for rotational capillary-gravity waves contained in \cite{Wah2}. 

All of the above existence results are heavily dependent on the absence of stagnation points for steady water waves. The condition which ensures the lack of stagnation points for a steady wave moving with constant speed $c$ is that all particles in the fluid have a horizontal velocity less than $c$. This is a physically reasonable assumption for water waves, without underlying currents containing strong non-uniformities, and which are not near breaking \cite{Light}.  Nevertheless, stagnation points are a very interesting phenomenon and ideally we would not wish to exclude them from our picture.  Stagnation points have long been a source of great interest and fascination in hydrodynamical research, dating back to Kelvin's work concerning ``cat's eyes'' and Stokes conjecture on the wave of greatest height (see \cite{Con1,Tol} for a thorough survey of Stokes' waves). Note that the physical meaning of a stagnation point is that of a fixed particle in the frame of reference moving with the wave-speed. However, 
there are stagnation points where $u=c,v=0$ but the wave does not actually stagnate there, for example the Stokes' wave of greatest height \cite{Con1,Tol}. Of course, allowing for the presence of stagnation points adds yet another major complication to the mathematical picture  (see the survey in \cite{CS3}) and we merely mention here that there have been two recent works \cite{ConVar,Wah3} where different approaches were used to show the existence of small amplitude waves with constant vorticity where stagnation points or critical layers occur. 

Mathematically rigorous work concerning the existence of small-amplitude waves for stratified flows which also admitted critical points was recently begun by one of the authors and his collaborators in \cite{EMM1}. Here the authors proved the existence of small amplitude waves for stratified flows which have a linear density distribution, and were also able to provide an explicit picture of the critical layers which were present in the resulting flows.
 In the following work, we prove the existence of small amplitude waves for stratified flows which have a more general density distribution function than that of \cite{EMM1}  and under conditions which could not be treated in 
\cite{Wa3}.

 The generality of the flows we admit complicates the analysis immensely, and the resulting governing equations which emerge take the form of an overdetermined semi-linear elliptic Dirichlet  problem. We then use and adapt various tools from the theory of elliptic equations and Fourier multiplier theory to prove, using the Crandall-Rabinowitz local bifurcation theorem, the existence of small amplitude water waves. 

The outline of the paper is as follows. In Section~\ref{W:1} we describe the Long-Yih \cite{Long,Turn,Yih} mathematical model for stratified capillary-gravity water waves, where we formulate  the governing equations equations in terms of the pseudo-streamfunction. 
Then, in Section~\ref{Lam}, we establish the existence of laminar flow solutions for the governing equations.
 Finally, in Section~\ref{Nonlam}, we prove, for a fixed volume of fluid, the existence of non-laminar flow solutions for the governing equations   using bifurcation theory and estimates based on the ellipticity of the system, culminating in our main results, Theorem ~\ref{MT1} and Theorem ~\ref{MT2}.

\section{The mathematical model} \label{W:1} 
We now present the Long-Yih \cite{Long,Turn,Yih} formulation of the governing equations for the motion of a two-dimensional inviscid, 
incompressible fluid with variable density. The equations will be formulated in terms of the velocity field $(u,v)$  of the fluid, $P$ 
the pressure distribution function, $\h$ the variable density function, and $g$ the gravitational constant of acceleration, in the fluid domain
\[
\0_\eta:=\{(x,y)\,:\,\text{$ x\in\s $ and $-1<y<\eta(t,x)$}\}.
\]
 The symbol $\s$ stands for the unit circle, and functions on $\s$ are identified with $2\pi-$periodic functions on $\R.$
 The function  $\eta$ describing the wave profile is assumed to satisfy $\eta(t,x)>-1$ for all $(t,x)$. For obvious reasons, we assume henceforth that the density $\h>0$.
 Taking $y = 0$ to represent the location of the undisturbed water surface we assume for any fixed time $t$ that
\begin{equation}\label{VC}
\int_\s\eta(t,x)\, dx=0.
\end{equation}
 
Restricting our investigations to steady travelling wave solutions of the governing equations  
enables us to express the problem in a moving reference frame where the flow is steady i.e. time independent, thereby simplifying the water wave problem.
We do this by supposing that there exists a positive constant $c$,
 called the {\em wave speed}, such that all functions  have an $(x,y,t)$ dependence of the form $(x-ct,y)$:
\begin{equation*}
\eta(t,x)=\eta(x-ct), \qquad (\h,u,v,P)(t,x,y)=(\h,u,v,P)(x-ct,y).
\end{equation*}
Under the assumption of constant temperature and zero viscosity, the motion of the fluid is prescribed by the steady state two-dimensional Euler equations 
\begin{subequations}\label{eq:S1}
\begin{equation}\label{eq:1}
\left\{
\begin{array}{rllllll}
 \h (u_x+ v_y)+\h_x(u-c)+\h_y v&=&0&\text{in}\ \0_\eta,\\
\h (u-c) u_x+\h vu_y&=&-P_x&\text{in}\ \0_\eta,\\
\h (u-c) v_x+\h vv_y&=&-P_y-g\h&\text{in}\ \0_\eta.\end{array}
\right.
\end{equation}
The appropriate boundary conditions for waves propagating on a fluid domain lying on an  impermeable flat bed for 
which surface tension plays a significant role are given by
\begin{equation}\label{eq:3}
\left\{
\begin{array}{lllllll}
v&=&(u-c)\eta'&\text{on}\ y=\eta(x),\\
P&=&P_0-\sigma\eta''/(1+\eta'^2)^{3/2}&\text{on} \ y=\eta(x),\\
v&=&0&\text{on}\ y=-1,
\end{array}
\right.
\end{equation}
with the constant $P_0$ being the atmospheric pressure and $\sigma$  the coefficient of surface tension. 
The first kinematic surface condition in  \eqref{eq:3} implies a non-mixing condition, namely that the wave surface consists of the same fluid particles 
for all times, while the last kinematic boundary condition in  \eqref{eq:3} ensures that the bottom of the ocean is impermeable.
 The contribution of the surface tension is felt on the free surface where it exerts a force which is proportional 
 to the curvature of the surface at any given point, the constant of proportionality being the coefficient of surface tension \cite{John,Light}.

Letting $V=(u-c,v)$ denote the velocity field in the moving frame of reference, the steady continuity equation simplifies to 
\begin{equation}\label{eq:2}
{\rm div} V=0\quad \text{in}\ \0_\eta
\end{equation} when we make the supposition that 
\begin{equation}\label{nondiff} V\cdot \nabla \h=0,
\end{equation}\end{subequations} a condition which ensures that the density $\h$ is non-diffusive \cite{Long,Yih,Turn}.
 This assumption is very important in our analysis, since it means that  $\nabla\cdot (\sqrt{\h} (u-c), \sqrt{\h} v)=0,$ and this relation enables us to define the associated  pseudo-streamfunction $\psi$ by
\begin{equation}\label{strmdef}
\text{$\p_x\psi=-\sqrt{\h} v \ $ and $ \ \p_y\psi=\sqrt{\h} (u-c)$\qquad  in $\ov\0_\eta$.}
\end{equation} 
The level sets of this function are indeed the streamlines of the steady flow: for suppose the path of a fluid particle in the steady flow 
is parametrised by the curve $(x(s),y(s))$, where $x'(s)=u(x,y)-c,y'(s)=v(x,y)$ then 
\[
\psi'(x(s),y(s))=\psi_x x'(s)+\psi_y y'(s)=0.
\] Therefore the function $\psi$ is constant along each streamline.
We can see directly that $\psi$ is defined for $\h>0$  by the expression
\[
\psi(x,y):=\lambda+\int_{-1}^{y}\sqrt{\h}(x,s)(u(x,s)-c)\, ds, \qquad (x,y)\in\ov\0_\eta.
\]
Since $\psi$ is constant on the free surface we have the freedom to choose $\lambda$ such that $\psi=0$ on $y=\eta(x).$ 
We then have 
\[
\lambda=\int_{-1}^{\eta(x)}\sqrt{\h}(x,s)(c-u(x,s))\, ds, 
\] where $\lambda$ is a constant called the \emph{mass flux}.
We note that a further assumption which is often invoked is that there  
are no stagnation points within the fluid domain, a condition we can express as   
\begin{equation}\label{eq:CD}
 \p_y\psi=\sqrt{\h}(u-c)<0.
\end{equation} This condition is known to be physically plausible for waves that are not near breaking (see for instance the discussions in \cite{CE3,John,Light}). 
For fluid motions where \eqref{eq:CD} holds we observe that \eqref{eq:CD} implies that $\lambda>0$.   
Mathematically, condition \eqref{eq:CD} is very useful because it enables us   to reexpress the water wave problem in terms of new $(q,p)$ variables
which transform the unknown fluid domain $\0_\eta$ into a fixed rectangular domain $\0_\lambda:=\s\times(-\lambda,0)$. 
Furthermore, streamlines in $\0_\eta$ are mapped to straight horizontal lines in the new domain $\0_\lambda$, and we see 
below that this provides us with a convenient method of analysing how certain quantities change either along fixed streamlines or as we vary the streamlines.  
The mapping is defined by the Dubreil-Jacotin \cite{Dub1} semi-hodograph transformation 
\begin{equation}\label{eq:Hod}
(q,p):=H(x,y):=(x,-\psi(x,y)),\qquad (x,y)\in\0_\eta.
\end{equation}
If \eqref{eq:CD} holds true then the mapping $H:\0_\eta\to\0_\lambda$, is an isomorphism which allows us to reformulate 
the problem in terms of the pseudostream function. We note that the non-diffusive condition on the density $\h$, \eqref{nondiff}, implies that
the density function is constant on each streamline, since \[
\p_q\left(\h\circ H^{-1}\right)\circ H =\h_x+\h_y\frac{v}{u-c}=\frac{(u-c)\h_x+v\h_y}{u-c}=0.
\]
This means that we can formulate the density as a function of $p$, $\ov\rho=\ov\h(p)$ where we set
$\ov\rho(p)=\h\circ H^{-1}(q,p)$ for all $p\in[-\lambda,0]$, and we call this new version $\ov\h$ of the density function  the {\em  streamline density function}.
Furthermore, Bernoulli's theorem states that the quantity 
\begin{equation}\label{Ber}
E:=P+\h\frac{(u-c)^2+v^2}{2}+g\h y 
\end{equation} is constant along streamlines,
 that is  $\p_q\left( E\circ H^{-1}\right)=0$ for fixed $p$. 
Particularly, when $p=0,$ we obtain 
\[
\frac{|\nabla \psi|^2}{2}-\sigma\frac{\eta''}{(1+\eta'^2)^{3/2}}+g\ov\rho(0)y=Q\qquad\text{on $y=\eta(x)$,}
\]
for some constant $Q\in\R$, where $Q$ is known as the \emph{hydraulic head} of the flow.
Moreover, a direct calculation shows that 
\[
\Delta\psi-gy\ov\rho'\circ H=-\p_p(E\circ H^{-1})\circ H\qquad\text{in $\0_{\eta}$}.
\]
This is the Long-Yih equation for steady stratified water waves \cite{Long,Turn,Yih}, and we infer that 
since $\p_q\left( E\circ H^{-1}\right)=0$ there exists a function $\beta=\beta(p)$, 
called {\em Bernoulli's function}, such that $-\p_p(E\circ H^{-1})=\beta$ in $\0_\lambda.$

Summarising, $\psi$ solves the following problem 
\begin{subequations}\label{eq:S2}
\begin{eqnarray}
\label{s21} \Delta \psi&=&f(y,\psi)\qquad \text{in}\ \0_\eta, 
\\ \label{s22} \psi&=&0 \ \qquad \qquad \text{on}\ y=\eta(x),
\\ \label{s23} \psi&=&\lambda \ \qquad\qquad \text{on} \ y=-1,
\\ \label{s24} \frac{|\nabla \psi|^2}{2}-\sigma\frac{\eta''}{(1+\eta'^2)^{3/2}}+g\ov\rho(0)y&=&Q \qquad \qquad\text{on} \ y=\eta(x),
\end{eqnarray}
where for convenience we define $f:[-1,1]\times\R\to \R$ to be the function 
\begin{eqnarray*}
f(y,\psi)&:=&gy\ov\h'(-\psi)+\beta(-\psi),\qquad  (y,\psi)\in[-1,1]\times\R.
\end{eqnarray*} 
More explicitly, we have $f=\sqrt \h [u_y-v_x]-{\h'}\left[(u-c)^2+ v^2 \right]/2$, and in the case of a homogeneous fluid
(i.e. constant density $\ov\h=\h$) then $f(y,\psi)=\beta(-\psi)=\sqrt{\h}\gamma(-\psi),$ with $\gamma$ denoting the 
vorticity function of the homogeneous fluid, and we get the usual governing equations for capillary-gravity flows with vorticity, \cite{Wah2,Wa3}.
We also note, from the divergence structure of the curvature operator,  that, fixing the volume of fluid to be equal 
to that when the fluid interface is at rest $\eta=0$, cf. \eqref{VC}, the head $Q$ can be determined from \eqref{s24}:
\begin{equation}\label{s25}
Q=\int_{\s}\frac{|\nabla \psi|^2}{2}(x,\eta(x))\, dx,
\end{equation}
 \end{subequations}
if we normalise the integral such that $\int_\s1\, dx=1.$
We observe that if we are given functions $\psi,\ov\h(-\psi),\beta(-\psi)$ which solve the Long-Yih equation \eqref{s21},  then the corresponding functions $(u,v,P)$ we obtain from the relations \eqref{strmdef}, \eqref{Ber} respectively, will satisfy equations \eqref{eq:1} and \eqref{eq:2}, see \cite{Turn}. 
This is irrespective of whether the non-stagnation condition \eqref{eq:CD} holds, and therefore in the following we make no such assumption on the fluid flow when studying the system \eqref{eq:S2}.

\section{Laminar flow solutions to system \eqref{eq:S2}}\label{Lam}
The stratified capillary-gravity water wave problem \eqref{eq:S2} is an over-determined semilinear Dirichlet system, with an additional boundary condition 
to be satisfied on the wave surface, which determines the effect of surface tension, given by equations \eqref{s24} and \eqref{s25}. 
In Theorem \ref{T:1} we prove that the semilinear system \eqref{s21}--\eqref{s23} is well-posed. 
In particular,  {this Theorem  is used to establish  the existence} of laminar-flow solutions to the full system \eqref{eq:S2}. 
 We then introduce an operator formulation of system \eqref{eq:S2} in Section~\ref{Nonlam} which allows us to address the question concerning whether
 non-trivial (that is, non-laminar) solutions exist which also satisfy the additional boundary conditions \eqref{s24}-\eqref{s25}.
 We do this by means of the Crandall-Rabinowitz local bifurcation theorem, which owing to its importance in subsequent developments we now state:  
\begin{thm}[\cite{CR}]\label{CR}
Let $X,Y$ be Banach spaces and let $\Mf\in C^k(\mathbb R \times X,Y)$ with $k\geq2$ satisfy:
\begin{enumerate}
\item[(a)] \label{cr1} $\Mf(\gamma,0)=0$ for all $\gamma\in\mathbb R$;
\item[(b)] \label{cr2} The Fr\'echet derivative $\p_x\Mf(\gamma^*,0)$ is a Fredholm operator of index zero with a one-dimensional kernel: \[
\Kern (\p_x\Mf(\gamma^*,0))=\{sx_0: s\in\mathbb R, 0\neq x_0\in X\};\]
\item[(c)] \label{cr3}
 The tranversality condition holds: 
\[
\p_{\gamma x}\Mf(\gamma^*,0)[(1,x_0)]\not \in \im(\p_x\Mf(\gamma^*,0)).
\]
\end{enumerate}
Then $\gamma^*$ is a bifurcation point in the sense that there exists $\epsilon_0>0$ and a branch of solutions 
\[
(\gamma,x)=\{(\Gamma(s),s\chi(s)): s\in \mathbb R, |s|<\epsilon_0\} \subset \mathbb R \times X,
\]
with $\Mf(\gamma,x)=0$, $\Gamma(0)=\gamma^*,\chi(0)=x_0$, and the maps 
\[
s\mapsto \Gamma(s)\in \mathbb R,\qquad  s\mapsto s\chi(s)\in X,
\] are of class $C^{k-1}$ on $(-\epsilon_0,\epsilon_0)$.
Furthermore there exists an open set $U_0\subset \R\times X$ with $(\gamma^*,0)\in U_0$ and 
\[
\{
(\gamma,x)\in U_0:\Mf(\gamma,x)=0, x\neq 0\}=\{(\Gamma(s),s\chi(s)): 0<|s|<\epsilon_0
\}.
\]
\end{thm}
Throughout the remainder of this paper we shall assume that the following conditions hold:
\begin{align*}
&\text{(A1)}\quad  \ov \h>0;\\
&\text{(A2)}\quad \ov \h\in C^{4-}(\R),    \beta\in C^{3-}(\R), \ \text{and} \   \beta, \ov \h'\in\mbox{\it BC}\,^2(\R);\\
&\text{(A3)}\quad \p_\psi f(y,\psi)\geq0 \ \text{for all} \ (y,\psi)\in[-1,1]\times\R.
\end{align*}
 Here, $\mbox{\it BC}\,^2(\R)$ is the subspace of $C^2(\R)$ which contains only functions with bounded derivatives up to order $2$, while, given $m\in\N $ with $m\geq1,$  
 \[
C^{m-}(\R):=\{h\in C^{m-1}(\R)\,:\, h^{(m-1)} \ \text{ is  Lipschitz continuous}\}.  
\]
\begin{rem}
Though it is possible to consider fluids having 
streamline density function which is not monotone, the assumption (A3) restricts our 
considerations to fluids having a non-increasing Bernoulli function. Furthermore, we mention that for the
particular case of fluids with constant density then condition (A3) is equivalent to the  condition $\gamma'(-\psi){ \leq} 0$ 
for the vorticity function. In this context, we note that this type of condition appears in studying (linear) stability properties of homogeneous
flows with a free surface, cf. the discussion in \cite{CS2}.
\end{rem}

  Let thus $\alpha\in(0,1)$ be fixed for the remainder of this paper, and given $m\in\N $ and $\eta>-1,$ 
we let $C^{m+\alpha}_{per}(\ov\0_\eta)$ denote the subspace of $C^{m+\alpha}(\ov\0_\eta)$ consisting only of functions which are $2\pi-$periodic in the $x-$variable.
\begin{thm}\label{T:1}
Assume that {\rm (A1)-(A3)}   are satisfied.
Given $\eta\in C^{2+\alpha}(\s)$  with $|\eta|<1$  and $\lambda\in\R,$ the semilinear Dirichlet problem 
\begin{equation}\label{eq:DP}
\left\{
\begin{array}{rllll}
\Delta \psi&=& f(y,\psi)&\text{in}\ \0_\eta,\\
\psi&=&0&\text{on}\ y=\eta(x),\\
\psi&=&\lambda&\text{on} \ y=-1
\end{array}
\right.
\end{equation}
possesses a unique solution $\psi\in C^{2+\alpha}_{per}(\ov\0_\eta).$
\end{thm}
We note that when $\eta=0$ and $\0:=\0_0$ then, for any $\lambda\in\R,$ Theorem \ref{T:1}
 ensures the existence and uniqueness of a solution $\psi_\lambda\in C^{2+\alpha}_{per}(\ov\0)$ of  problem \eqref{eq:DP}, 
which in this case depends only upon $y$.
 Therefore, the pair $(\eta,\psi):=(0,\psi_\lambda)$ is a solution of \eqref{eq:S2} for all $\lambda\in\R$.
Moreover, we obtain that the function $\psi_\lambda$ verifies the following identity
\begin{equation}\label{eq:PL}
\psi_\lambda(y)=-\lambda y-y\int_{-1}^0\int_{-1}^tf(s,\psi_\lambda(s))\, ds\, dt-\int_{y}^0\int_{-1}^tf(s,\psi_\lambda(s))\, ds\, dt,\qquad y\in[-1,0]. 
\end{equation}
Particularly, by differentiation we have
\begin{equation}\label{eq:PLd}
\psi_\lambda'(y)=-\lambda -\int_{-1}^0\int_{-1}^tf(s,\psi_\lambda(s))\, ds\, dt+\int_{-1}^yf(s,\psi_\lambda(s))\, ds,\qquad y\in[-1,0].  
\end{equation}

\begin{proof}[Proof of Theorem \ref{T:1}] We prove the uniqueness first. 
If $\psi_1,\psi_2\in C^{2+\alpha}_{per}(\ov\0_\eta)$ are two solutions of \eqref{eq:DP}, then $\psi:=\psi_1-\psi_2$ is a solution of the equation
\[
\Delta \psi=f(y,\psi_1)-f(y,\psi_2)\qquad\text{in} \ \0_\eta
\] 
satisfying $\psi=0$ on $\p\0_\eta.$ 
Using the mean value theorem, we have:
\[
\text{$f(y,\psi_1)-f(y,\psi_2)=c(x,y)\psi\quad $ with $\quad c(x,y):=\int_0^1\p_\psi f(y,\psi_2+t(\psi_1-\psi_2))\, dt$}.
\]
Invoking (A3), the weak elliptic maximum principle yields the desired uniqueness result.

In order to prove the existence part, we use a fixed point argument.
Given $\varphi\in C^\alpha_{per}(\ov\0_\eta),$ we denote by $T\varphi\in C^{2+\alpha}_{per}(\ov\0_\eta)$ the unique solution of the linear Dirichlet problem
\begin{equation}\label{eq:DPv}
\left\{
\begin{array}{rllll}
\Delta \psi&=&f(y,\varphi)&\text{in}\ \0_\eta,\\
\psi&=&0&\text{on}\ y=\eta(x),\\
\psi&=&\lambda&\text{on} \ y=-1.
\end{array}
\right.
\end{equation}
By (A2), the operator $T$ is well-defined and continuous.
Moreover,   $C^{2+\alpha}_{per}(\ov\0_\eta)\hookrightarrow C^\alpha_{per}(\ov\0_\eta)$ is a compact embedding,
and  $T: C^\alpha_{per}(\ov\0_\eta)\to C^\alpha_{per}(\ov\0_\eta)$ is therefore completely continuous.
Since any solution of \eqref{eq:DP} is a fixed point of $T$, Leray-Schauder's theorem \cite[Theorem 10.3]{GT} yields the following result.
\begin{itemize}
\item[] \it If there exists  a  constant $M>0$ and $\tilde\alpha\in (0,1)$ such that 
\begin{equation}\label{eq:CE}
\|\psi\|_{C^{\tilde\alpha}_{per}(\ov\0_\eta)}\leq M
\end{equation} 
for all $\psi\in C^{2+\alpha}_{per}(\ov\0_\eta)$ and $\kappa\in[0,1]$ satisfying $\psi=\kappa T\psi,$
then $T$ possesses a fixed point. 
\end{itemize}
That $\psi$ is a solution of $\psi=\kappa T\psi,$ may be re-expressed by saying that $\psi$ is a solution of the Dirichlet problem:
\begin{equation}\label{eq:DPp}
\left\{
\begin{array}{rllll}
\Delta \psi&=&\kappa f(y,\psi)&\text{in}\ \0_\eta,\\
\psi&=&0&\text{on}\ y=\eta(x),\\
\psi&=&\kappa\lambda&\text{on} \ y=-1.
\end{array}
\right.
\end{equation}
We prove now \eqref{eq:CE}.
 First, we  bound the supremum bound of $\psi$ independently of $\kappa.$
Indeed,  using apriori bounds for inhomogeneous Dirichlet problems, cf. \cite[Theorem 3.7]{GT},
 we obtain, by taking into account that $\0_\eta\subset [|y|<1],$ the following estimate 
\begin{equation}\label{eq:C1}
\|\psi\|_{C_{per}(\ov\0_\eta)}\leq \sup_{\p\0_\eta} |\psi|+e^{2}\sup_{\0_\eta}|\kappa f(y,\psi)|\leq C,
\end{equation}
where $C:=|\lambda|+e^{2}\left(g\|\ov\h'\|_{BC(\R)}+\|\beta\|_{BC(\R)}\right)$.
We may regard now the right-hand side of the first equation in \eqref{eq:DPp} as a function in $L_\infty(\0_\eta).$ 
Then, taking into account that the Dirichlet data in \eqref{eq:DPp} are constant, we infer from the de Giorgi and Nash estimates, cf. \cite[Theorem 8.29]{GT},
that  there exists $\tilde\alpha\in(0,1)$ such that equation \eqref{eq:CE} is satisfied.
\end{proof}

\section{Non-laminar solutions}\label{Nonlam}
\subsection{The functional analytic setting} 
The next aim is to  show that there exist non-trivial (that is, non-laminar) solutions of the stratified water wave problem \eqref{eq:S2}, and in order 
to achieve this we must recast the water wave problem \eqref{eq:S2} using an operator formulation. We will then apply the Crandall-Rabinowitz theorem 
to prove the local existence of solutions bifurcating from the laminar flows. We can show the local existence of non-laminar solutions, firstly fixing 
the mass-flux constant $\lambda$ and then using the coefficient of surface tension $\sigma$ as a bifurcating parameter, and then secondly fixing $\sigma$ 
and using $\lambda$ as a bifurcating parameter. 

To this end, we introduce the subspaces $\wh C_{e,k}^{m+\alpha}(\s),$ $m\in\N$,  of $C^{m+\alpha}(\s)$ consisting  of even functions which are 
$2\pi/k-$periodic { and have integral mean zero}, and 
 we define $C_{e,k}^{m+\alpha}(\ov\0),$  as the subspace of $C^{m+\alpha}_{per}(\ov\0)$ containing only  even and $2\pi/k-$periodic functions  
 in the $x$ variable.
The evenness condition imposes a symmetry on the free-surface (and on the underlying flow). 
In the absence of stratification and of stagnation points (but allowing for vorticity), one can show that the mere monotonicity of the
free-surface between trough and crest ensures this symmetry, cf. \cite{CEW,CE2}.
To some extent, this result also applies to stratified flows, cf. \cite{Wa2}.
By  considering only functions $\eta $ with integral mean zero we incorporate the volume constraint \eqref{VC} in the spaces we work with.

The bottom of the fluid being located at $y=-1,$  we take the wave profile $\eta$ from the set
\[
\Ad:=\{\eta\in \wh C_{e,k}^{2+\alpha}(\s)\,:\, |\eta|<1\}.
\]
Given $\eta\in\Ad,$ we define the mapping $\Phi_\eta:\0=\s\times(-1,0)\to \0_\eta$
by the relation
\[
\Phi_\eta(x,y):=(x,(y+1)\eta(x)+y), \qquad (x,y)\in\0.
\] 
This mapping has the benefit of flattening the unknown free-boundary fluid domain into a fixed rectangle, similar to \cite{EMM1,Wah3}.
The operator $\Phi_\eta$ is a diffeomorphism for all $\eta\in\Ad$,   and we use this property to transform problem \eqref{eq:S2} onto the rectangle $\0.$
Therefore, corresponding to the semilinear elliptic operator in \eqref{s21} we introduce the transformed
elliptic operator  $\A:\Ad\times C_{e,k}^{2+\alpha}(\ov\0)\to C_{e,k}^{\alpha}(\ov\0)$  with
\[
\A(\eta,\tilde\psi):=\Delta(\tilde\psi\circ\Phi_\eta^{-1})\circ\Phi_\eta-f(y,\tilde\psi\circ\Phi_\eta^{-1})\circ\Phi_\eta,
\qquad (\eta,\tilde\psi)\in\Ad\times C_{e,k}^{2+\alpha}(\ov\0).
\]
To ease  notation we decompose $\A(\eta,\tilde\psi)=\A_0(\eta)\tilde\psi+b(\eta,\tilde\psi),$ where
\begin{equation}\label{eq:A0}
\begin{aligned}
\A_0(\eta)&:=\p_{11}-\frac{2(1+y)\eta'}{1+\eta}\p_{12}+\frac{1+(1+y)^2\eta'^2}{(1+\eta)^2}\p_{22}-(1+y)\frac{(1+\eta)\eta''-2\eta'^2}{(1+\eta)^2}\p_2,\\
b(\eta,\tilde\psi)&:=-f((1+y)\eta+y,\tilde \psi).
\end{aligned}
\end{equation}
Furthermore, corresponding to the equation \eqref{s24}, we introduce the  boundary operator $\B:\Ad\times C_{e,k}^{2+\alpha}(\ov\0)\to C_{e,k}^{\alpha}(\s)$ by the relation
\[
\B(\eta,\tilde\psi):=\frac{\tr |\nabla(\tilde\psi\circ\Phi_\eta^{-1})|^2\circ\Phi_\eta}{2}=\frac{1}{2}\left(\tr \tilde\psi_1^2-\frac{2\eta'}{1+\eta}\tr \tilde\psi_1\tr \tilde\psi_2+\frac{1+\eta'^2}{(1+\eta)^2}\tr \tilde\psi_2^2\right),
\]
with $\tr $ denoting the trace operator with respect to $\s=\s\times\{0\}$.

Having introduced  this notation, we remark that the laminar flow solutions  
 $(\eta,\psi)=(0,\psi_\lambda), $ $\lambda\in\R,$  correspond to the trivial solutions $(\sigma, \lambda, \eta)=(\sigma, \lambda,0) $ 
 of the  equation
\begin{equation}\label{eq:Psi}
\Psi(\sigma,\lambda,\eta)=0
\end{equation}
 where $\Psi:(0,\infty)\times \R\times\Ad\to \wh C_{e,k}^{\alpha}(\s)$ is given by
 \[
\Psi(\sigma,\lambda,\eta):=\B(\eta, \T(\lambda,\eta))-\sigma\frac{\eta''}{(1+\eta'^2)^{3/2}}+g\ov\h(0)\eta-\int_{\s}\B(\eta, \T(\lambda,\eta))\, dx,
\]
and $\T:\R\times\Ad\to C_{e,k}^{2+\alpha}(\ov\0)$ is the solution operator to the semilinear Dirichlet problem
\begin{equation}\label{eq:DPp-}
\left\{
\begin{array}{rllll}
\A(\eta,\tilde\psi)&=&0&\text{in}&\0,\\
\tilde\psi&=&0&\text{on}&y=0,\\
\tilde\psi&=&\lambda&\text{on} &y=-1.
\end{array}
\right.
\end{equation}
That $\T$ is well-defined between these spaces follows easily from the weak elliptic maximum principle, as in the uniqueness proof of Theorem \ref{T:1}.
Moreover, problems \eqref{eq:S2} and \eqref{eq:Psi} are  equivalent, { that is  $(\lambda,\sigma,\eta)$ is a solution of \eqref{eq:Psi} if and only if $(\eta,\psi:=\T(\lambda,\eta)\circ\Phi_\eta^{-1})$
is a solution of \eqref{eq:S2}.}
Following ideas from \cite[Theorem 2.3]{EMM2}, we establish now a preliminary result concerning the regularity of $\T.$

 \begin{lemma}\label{L:1} We have  $\T\in C^2(\R\times \Ad, C_{e,k}^{2+\alpha}(\ov\0)).$
 \end{lemma}
\begin{proof}
Letting $L: \R\times\Ad\times C^{2+\alpha}_{e,k}(\ov\0)\to C^\alpha_{e,k}(\ov\0)\times (C^{2+\alpha}_{e,k}(\s))^2$ be the operator defined by 
\[ L(\lambda,\eta,\tilde\psi):=(\A(\eta,\tilde\psi), \tr_0 \tilde\psi,\tr_{-} \tilde\psi-\lambda),\]
where $\tr_- $ is the trace  with respect to $\s=\s\times\{-1\}$, the assumption (A2) guarantees  that 
$L$ is of class $C^2$.
Moreover, it holds that 
\[L(\lambda, \eta, \T(\lambda,\eta))=0\]
for all $(\lambda,\eta)\in\R\times \Ad.$
We show now that the operator  $\p_{\tilde\psi}L(\lambda, \eta, \T(\lambda,\eta))$ is an isomorphism, that is
$\p_{\tilde\psi}L(\lambda, \eta, \T(\lambda,\eta))\in\Isom(C^{2+\alpha}_{e,k}(\ov\0), C^\alpha_{e,k}(\ov\0)\times (C^{2+\alpha}_{e,k}(\s))^2),$
for all $(\lambda,\eta)\in\R\times\Ad.$
The implicit function theorem leads us then to the claim.
It is not difficult to see that for all $\tilde\psi, w\in C^{2+\alpha}_{e,k}(\ov\0),$ we have
\[
\p_{\tilde\psi}L(\lambda,\eta,\tilde\psi)[w]=(\A_0(\eta)w+\p_{\tilde\psi}b(\eta,\tilde\psi)w,\tr_0 w,\tr_-w), 
\]
where
\[
\p_{\tilde\psi}b(\eta,\tilde\psi)=-\p_{\psi}f((1+y)\eta+y,\tilde\psi)\leq0, 
\]
by (A3).  
Consequently, 
 the operator $\p_{\tilde\psi} L(\lambda,\eta,\tilde\psi)[w]$ satisfies the weak-maximum principle \cite{GT}, and using standard 
 arguments for elliptic operators on the existence and uniqueness of solutions it follows that $\p_{\tilde\psi} L(\lambda,\eta,\tilde\psi)[w]$ 
 is an isomorphism for all $(\lambda,\eta,\tilde\psi).$
This completes the proof.
\end{proof}

Since  $\B$ depends analytically on $(\eta,\tilde\psi)$, we conclude that  
\begin{equation}\label{eq:B0}
\Psi\in C^2((0,\infty)\times \R\times\Ad, \wh C_{e,k}^{\alpha}(\s)).
\end{equation}
This is the starting point to our bifurcation analysis.

\subsection{Caracterisation of $\p_\eta\Psi(\sigma,\lambda,0)$ as a Fourier multiplier}
In order to prove that local bifurcation occurs using the Crandall-Rabinowitz Theorem \ref{CR} we need to determine first the 
Fr\'echet derivative of $\Psi$ with respect to $\eta.$ In this section, we show that $\p_\eta\Psi(\sigma,\lambda,0)$ is in fact a 
Fourier multiplier, and this allows us to analyse its Fredholm properties in a rather neat fashion by examining the behaviour of its 
symbols $(\mu_m(\sigma,\lambda))$, which we determine below. This examination of the symbols will be the essence of the technical proofs 
of Theorems \ref{MT1}, \ref{MT2} in Section \ref{MainResults}. 
To this end we fix $(\sigma,\lambda)$ and, recalling that $\T(\lambda,0)=\psi_\lambda,$ we  obtain
\begin{align*}
\p_\eta\Psi(\sigma,\lambda,0)[\eta]=&\p_\eta\B(0,\psi_\lambda)[\eta]+\p_{\tilde\psi}\B(0,\psi_\lambda)[ \p_\eta \T(\lambda,0)[\eta]]-\sigma\eta''+g\ov \h(0)\eta\\
&-\int_\s\left(\p_\eta\B(0,\psi_\lambda)[\eta]+\p_{\tilde\psi}\B(0,\psi_\lambda)[ \p_\eta \T(\lambda,0)[\eta]]\right)\, dx
\end{align*}
for all $\eta \in \wh C^{2+\alpha}_{e,k}(\s)$.
Taking into account that $\psi_\lambda$ depends only upon $y$, we have the following relations:
\begin{align*}
\p_\eta\B(0,\psi_\lambda)[\eta]&=-\psi_\lambda'^2(0)\eta\quad\text{and}\quad\p_{\tilde\psi}\B(0,\psi_\lambda)[w]=\psi_\lambda'(0)\tr w_2.
\end{align*}
Moreover, differentiating  the equations of \eqref{eq:DPp-} with respect to $\eta,$ we see that $\p_\eta\T(\lambda,0)[\eta]$
is the solution of the Dirichlet problem:
\begin{equation}\label{eq:FD1}
\left\{
\begin{array}{rllll}
\A_0(0) w+\p_{\tilde\psi}b(0,\psi_\lambda)w&=&-\p_\eta\A_0(0)[\eta]\psi_\lambda-\p_\eta b(0,\psi_\lambda)[\eta]&\text{in}\ \0,\\[1ex]
w&=&0&\text{on}\ \p\0,
\end{array}
\right.
\end{equation}
whence
\begin{equation}\label{eq:FD}
\left\{
\begin{array}{rllll}
\Delta w-\p_\psi f(y,\psi_\lambda)w&=&(1+y)\psi_\lambda'\eta''+[2f(y,\psi_\lambda)+g(1+y)\ov\h'(-\psi_\lambda)]\eta&\text{in}\ \0,\\[1ex]
w&=&0&\text{on}\ \p\0.
\end{array}
\right.
\end{equation}
We consider now the Fourier series expansions of $\eta$ and $w$: 
\[
\eta=\sum_{m=1}^\infty a_m\cos(mkx),\qquad w:=\sum_{m=1}^\infty a_mw_m(y)\cos(mkx),
\] 
and find that   
  \[
 (1+y)\psi_\lambda'\eta''+[2f(y,\psi_\lambda)+g(1+y)\ov\h'(-\psi_\lambda)]\eta=\sum_{m=1}^\infty a_mb_m(y)\cos(mkx)
 \]
where
\[
b_m(y):=2f(y,\psi_\lambda(y))+g(1+y)\ov\h'(-\psi_\lambda(y))-(mk)^2(1+y)\psi_\lambda'(y).
\]
Inserting these expressions into \eqref{eq:BvD} and comparing the coefficients of $\cos(mkx) $ on both sides of the equations, we conclude that $w_m$
is the unique solution of the boundary value problem
\begin{equation}\label{eq:BvD}
\left\{
\begin{array}{rllll}
 w_m''-((mk)^2+\p_\psi f(y,\psi_\lambda))w_m&=&b_m,\ -1<y<0,\\[1ex]
w_m(0)=w_m(-1)&=&0.
\end{array}
\right.
\end{equation}
Though we can not solve \eqref{eq:BvD} explicitly, we will see later on that elliptic maximum principles may be use to study the properties of the solution of  \eqref{eq:BvD}.
Understanding the behaviour of the solutions of  \eqref{eq:PL}   and \eqref{eq:BvD} is the key point in our analysis.
   
Summarising, we find that the Fr\'echet derivative of $\p_\eta\Psi(\sigma,\lambda,0)$ is the Fourier multiplier
 \begin{equation}\label{eq:FDd}
\p_\eta\Psi(\sigma,\lambda,0)\sum_{m=1}^\infty a_m\cos(mkx)=\sum_{m=1}^\infty \mu_m(\sigma,\lambda) a_m\cos(mkx),
\end{equation}
 and the symbol $(\mu_m(\sigma,\lambda))_{m\geq1}$ is given by the expression
 \begin{equation}\label{eq:symb}
\mu_m(\sigma,\lambda):=\sigma(mk)^2+g\ov\h(0)-\psi'^2_\lambda(0)+\psi_\lambda'(0)w_m'(0).
\end{equation}
Clearly, by virtue of Lemma \ref{L:1}, $\mu_m$ is of class $C^2$ in $(\sigma,\lambda)$.
It is well-known that realization $\Delta:C^{2+\alpha}(\s)\to C^{\alpha}(\s)$  of the Laplace operator is the generator of an analytic and continuous semigroup in $\kL(C^{\alpha}(\s)).$
Moreover, $\Delta$ is a Fourier multiplier with symbol  $(-k^2)_{k\in\Z}$
and the operator $\sigma \Delta+\p_\eta\Psi(\sigma,\lambda,0)\in \kL(\wh C^{2+\alpha}_{e,k}(\s), \wh  C^{1+\alpha}_{e,k}(\s)),$ cf. Lemma \ref{L:1}.
This shows, in view of Proposition 2.4.1 (i) in \cite{L}, that $-\p_\eta\Psi(\sigma,\lambda,0)$ is itself the generator of an analytic strongly continuous semigroup in $\kL(\wh C_{e,k}^{\alpha}(\s)).$
Particularly, we infer from Theorem III.8.29 in \cite{K} that the spectrum of $\p_\eta\Psi(\sigma,\lambda,0)$ coincides with its point spectrum, that is
\begin{equation}
\bm{\sigma}(\p_\eta\Psi(\sigma,\lambda,0)):=\{\mu_m(\sigma,\lambda)\,:\, m\geq 1\}.
\end{equation}
 
\subsection{Main existence results}\label{MainResults}
We now prove our main existence results, first using the coefficient of surface tension $\sigma$ as a bifurcation parameter (the $\gamma$ in the Crandall-Rabinowitz Theorem \ref{CR}) and secondly using the mass-flux $\lambda$ as a bifurcation parameter. We make the following additional assumption on the system \eqref{eq:S2}:
\begin{equation*}
\text{(A4)}\quad  2f(y,\psi)+g(1+y)\ov\h'(-\psi)\leq 0 \ \text{for all} \ y\in[-1,0]\times(-\infty,0]. 
\end{equation*}
\begin{rem}
We note that for homogeneous fluids, $\h$ constant, condition (A4) states that the vorticity function is negative, $\gamma(-\psi)\leq0$. For an analysis of the bearing that such a condition has on the underlying flows cf. the discussions in \cite{CS3,Var}. 
\end{rem}
\begin{thm}\label{MT1} Let $\ov \h$ and $\beta$ be given such that {\rm (A1)-(A4)} are satisfied and fix  $k\geq 1$.
There exists $\Lambda_-\in\R$ such that for all $\lambda\leq\Lambda_-,$ we may find a sequence $(\ov \sigma_m)_{m\geq1}\subset(0,\infty)$ decreasing to zero with the following properties: 
\begin{itemize}
\item[$(i)$] Given $m\geq1 $,  there exist a  curve
$(\sigma_m,\eta_m):(-\e,\e)\to(0,\infty)\times \Ad$ which is continuously  differentiable and consists only of solutions of the problem \eqref{eq:Psi}, 
that is $\Psi(\sigma_m(s),\lambda,\eta_m(s))=0$
for all $|s|<\e.$ 
\item[$(ii)$] We have the following asymptotic  relations
 \[\text{$\sigma_m(s)=\ov\sigma_m+O(s),\quad \eta_m(s)=-s\cos(mkx)+O(s^2)$ \qquad for $s\to0.$} \]
\end{itemize} 
Moreover, for fixed $\lambda\leq\Lambda_-,$ all the solutions of \eqref{eq:Psi} close to $(\ov\sigma_m,\lambda,0)$ are either laminar flows or belong to  the curve $(\sigma_m,\lambda,\eta_m).$
\end{thm}
\begin{rem} \label{R:1} If we replace (A4) by
\begin{itemize}
\item[(A4$'$)] $2f(y,\psi)+g(1+y)\ov\h'(-\psi)\geq 0$  for all $(y,\psi)\in[-1,0]\times[0,\infty)$
\end{itemize}
then the conclusion of Theorem \ref{MT1} remains true with the modification that in this case we find 
$\Lambda_+\in\R$ such that the assertions $(i)$ and $(ii)$ of the theorem hold  for all   $\lambda\geq\Lambda_+.$
 In this case we obtain travelling wave solutions of the original problem \eqref{eq:S1}, which have positive mass flux.
\end{rem}

For our second main result  we  require, in addition to {\rm (A1)-(A4)}, that 
\begin{align*}
&\text{(B1)}\quad  \p_{\psi\psi}f\geq0 \ \text{on} \ [-1,0]\times(-\infty,0];\\
&\text{(B2)}\quad 2\p_\psi f(y,\psi)-g(1+y)\ov\h''(-\psi)\geq 0  \ \text{for  all} \  (y,\psi)\in[-1,0]\times(-\infty,0];\\
&\text{(B3)}\quad \p_{\psi} f(y,0)\leq 2  \ \text{for  all} \  y\in[-1,0].
\end{align*}
\begin{rem}
For the particular setting of homogeneous flows ($\h$ constant), we can see that (B1) requires that the vorticity function $\gamma''(-\psi)\geq 0$, (B2) is simply equivalent to (A3), and (B3) requires that on the surface $\gamma'(0)\geq -2$.  
\end{rem}

\begin{thm}\label{MT2} Let $\ov \h$ and $\beta$ be given such that {\rm (A1)-(A4)} and {\rm (B1)-(B3)} are satisfied and let $\sigma>0$ be fixed.
There exists a positive integer $K\in\N$ and for all $k\geq K$  a sequence $(\ov \lambda_m)_{m\geq1}\subset\R$ with $\ov\lambda_m\to-\infty$ and: 
\begin{itemize}
\item[$(i)$] Given $m\geq1 $,  there exist a curve $(\lambda_m,\eta_m):(-\e,\e)\to\R\times \Ad$  which is
 continuously  differentiable and consists only of solutions of the problem \eqref{eq:Psi}, that is $\Psi(\sigma,\lambda_m(s),\eta_m(s))=0$
for all $|s|<\e.$ 
\item[$(ii)$] We have the following asymptotic  relations
 \[\text{$\lambda_m(s)=\ov\lambda_m+O(s),\quad \eta_m(s)=-s\cos(mkx)+O(s^2)$ \qquad for $s\to0.$} \]
\end{itemize} 
Moreover,  all the solutions of \eqref{eq:Psi} close to $(\sigma,\ov\lambda_m,0)$ are either laminar flows or belong to  the curve $(\sigma,\lambda_m,\eta_m).$
\end{thm}
\begin{rem}
We note that for $s$ sufficiently small, it follows directly from Theorem \ref{MT1} $(ii)$ and Theorem \ref{MT2} $(ii)$ that the resulting non-laminar solutions have a wave surface profile $\eta_m$  which has minimal period $2\pi/mk$, $\eta_m$ has a unique crest and trough per period, and $\eta_m$ is strictly monotone from crest to trough.
\end{rem}
  
Before proving the theorems, we consider more closely the functions $\psi_\lambda.$
Since  by Lemma \ref{L:1} the mapping $\lambda\mapsto\psi_\lambda$ is of class $C^2,$ we differentiate the equations of \eqref{eq:DPp-}, when $\eta=0,$ with respect to $\lambda$ and find that 
$\p_\lambda\psi_\lambda$ is the solution of the Dirichlet problem:
 \begin{equation}\label{eq:Pl}
\left\{
\begin{array}{rllll}
\Delta u&=& \p_\psi f(y,\psi_\lambda)u&\text{in}\ \0,\\
u&=&0&\text{on}\ y=0,\\
u&=&1&\text{on} \ y=-1.
\end{array}
\right.
\end{equation}
Invoking (A3), we obtain that $\p_\lambda\psi_\lambda(0)\leq \p_\lambda\psi_\lambda\leq\p_\lambda\psi_\lambda(-1)$, and since $\p_\lambda\psi_\lambda$ is not constant
 Hopf's principle ensures that 
\[
\text{$\p_\lambda\psi'_\lambda(0)<0\quad \text{and}\quad \p_\lambda\psi_\lambda'(-1)<0$ \qquad for all $\lambda\in\R.$}
\] 
Consequently, 
 $\lambda\mapsto \psi'_\lambda(0)$
 is a  decreasing function and, recalling (A2) and \eqref{eq:PLd}, we obtain  that and there exists a {threshold value} $\Lambda\in\R$ with 
\begin{equation}\label{eq:bub}
\text{$\psi'_\Lambda(0)=0$} \quad  \text{and} \quad
\left\{\begin{array}{cccc}
\psi_\lambda'(0)>0,&\text{for} \ \lambda<\Lambda,\\[1ex]
\psi_\lambda'(0)<0,&\text{for} \ \lambda>\Lambda.
\end{array}
\right. 
\end{equation}
\begin{proof}[Proof of Theorem \ref{MT1}]
Fix $k\in\N, k\geq 1.$ 
For each $m\in\N,$ $m\geq1,$ and $\lambda\in\R$, we let
\begin{equation}\label{eq:BE}
 \ov\sigma_m(\lambda)=:\ov\sigma_m:=-\frac{g\ov\h(0)-\psi'^2_\lambda(0)+\psi_\lambda'(0)w_m'(0)}{(mk)^2},
 \end{equation} 
be the solution of the equation $\mu_m(\lambda,\cdot)=0$.

We first show that if $\lambda$ is small, then $\ov\sigma_m$ are all positive.
To this end, we claim that the right hand side $b_m$ of \eqref{eq:BvD}$_1$ is non-positive  for all $m\in\N$, provided $\lambda$ is small enough.
Indeed, recalling  \eqref{eq:PLd}, we see that $\psi_\lambda'\to_{\lambda\to-\infty}\infty$ uniformly in $y\in[-1,0]$, so that
 there exists $\Lambda_-\leq\Lambda$ with the following properties:
\begin{equation}\label{L-}
\text{  $\psi_\lambda'\geq 0 \qquad $ and $\qquad  g\ov\h(0)\leq\psi'^2_\lambda(0)$ \quad for all $\lambda\leq\Lambda_-.$}
\end{equation}
The claim follows now from  (A4) 
  and the fact that $\psi_\lambda\leq 0$ for  all $\lambda\leq \Lambda_-.$ 
Let thus $\lambda\leq\Lambda_-$ be fixed. 
The weak  elliptic maximum principle applied to \eqref{eq:BvD} ensures that $w_m$ attains its minimum on the boundary, that is $w_m\geq0$, and  Hopf's principle leads us to $w_m'(0)<0$ for all $m\in\N.$
Consequently, $\ov\sigma_m>0$ for all $m\geq 1$ and $\mu_m(\lambda,\ov \sigma_m)=0$.

Next we prove that   $(\ov\sigma_m)$ is a decreasing sequence.
Recalling \eqref{eq:BE}, it suffices to show that $\left( w_m'(0)/(mk)^2\right)_{m\geq1}$
 is  non-decreasing.
In order to do this, we consider the parameter dependent problem
\begin{equation}\label{eq:mu}
\left\{
\begin{array}{rllll}
 u_\chi''-(\chi+\p_\psi f(y,\psi_\lambda))u_\chi&=&\displaystyle\frac{2f(y,\psi_\lambda)+g(1+y)\ov\h'(-\psi_\lambda)}{\chi}-(1+y)\psi_\lambda',\ -1<y<0,\\[1ex]
u_\chi(0)=u_\chi(-1)&=&0.
\end{array}
\right.
\end{equation}
which is obtained by replacing $(mk)^2$ in \eqref{eq:BvD} by $\chi\in(0,\infty) $ and dividing by this variable the equations of \eqref{eq:BvD}.
Particularly, $w_m=(mk)^2u_{(mk)^2} $, and the arguments presented above show that $u_\chi\geq0$ for all $\chi>0.$ 
Clearly, the solution $u_\chi$ of  \eqref{eq:mu} depends smoothly upon $\chi.$
Differentiating the equations in \eqref{eq:mu} with respect to $\chi$ yields that  $v_\chi:=\p_\chi u_\chi$ is the solution of
\begin{equation}\label{eq:mumu}
\left\{
\begin{array}{rllll}
 v_\chi''-(\chi+\p_\psi f(y,\psi_\lambda))v_\chi&=&\displaystyle u_\chi-\frac{2f(y,\psi_\lambda)+g(1+y)\ov\h'(-\psi_\lambda)}{\chi^2} \geq0, &-1<y<0,\\[1ex]
v_\chi(0)=v_\chi(-1)&=&0.
\end{array}
\right.
\end{equation}
Whence, $v_\chi\leq0$ and Hopf's principle implies that $v_\chi'(0)=\p_\chi(u_\chi'(0))>0$ for all $\chi\geq1.$
Particularly, the sequence $\left(w_m'(0)/(mk)^2\right)_{m\geq1}$ is  increasing.

We have thus shown that $\mu_m(\ov\sigma_p,\lambda)=0$ if and only if $m=p\geq1.$ 
Consequently, we have 
\begin{equation}\label{eq:B1}
\Kern \p_\eta\Psi(\ov\sigma_m,\lambda,0)=\spn\{\cos(mkx)\},\quad  \im \p_\eta\Psi(\ov\sigma_m,\lambda,0)\oplus \{\cos(mkx)\}=\wh C^{\alpha}_{e,k}(\mathbb S),
\end{equation}and so $\p_\eta\Psi(\ov\sigma_m,\lambda,0)$ is a Fredholm operator of index zero.
Moreover, in virtue of \eqref{eq:symb}, the transversality condition \eqref{cr3} of Theorem \ref{CR} is fulfilled too
\begin{equation}\label{eq:B2}
\p_{\sigma\eta}\Psi(\ov\sigma_m,\lambda,0)[\cos(mkx)]=2(mk)^2\cos(mkx)\notin\im \p_\eta\Psi(\ov\sigma_m,\lambda,0).
\end{equation}
Gathering \eqref{eq:B0}, \eqref{eq:B1}, and \eqref{eq:B2}  we may apply Theorem \ref{CR} to equation \eqref{eq:Psi} and obtain the desired existence result.  

We finish by proving that $(\overline \sigma_m)$ converges to zero.
We apply the estimate from Theorem 3.7 in \cite{GT} to the Dirichlet problem \eqref{eq:mu}, when $\chi=(mk)^2,$
and  obtain that
\begin{equation}\label{eq:bw_p2}
\sup_m\|w_m/(mk)^2\|_{C([-1,0])}<\infty.
\end{equation}
Letting $u_m:=w_m/(mk)^2,$ we re-write \eqref{eq:BvD} as follows
\begin{equation}\label{eq:va}
\left\{
\begin{array}{rllll}
 u_m''-(mk)^2u_m&=&B_m(y),&-1<y<0,\\
u_m(0)=u_m(-1)&=&0,
\end{array}
\right.
\end{equation}
where $B_m(y):=(\p_\psi f(y,\psi_{\lambda}(y))w_m(y)+b_m(y))/(mk)^2$ is, by \eqref{eq:bw_p2}, a bounded sequence in $C([0,1]).$
The general solution of \eqref{eq:va} is
\[u_m(y)=\frac{\sinh(mky)}{\sinh({mk})}\int_{-1}^0\frac{\sinh({mk}(s+1))}{{mk}}B_m(s)\, ds+\int_0^y\frac{\sinh({mk}(y-s))}{mk}B_m(s)\, ds,\]
and particularly 
\[u_m'(0)=\frac{mk}{\sinh(mk)}\int_{-1}^0\frac{\sinh(mk(s+1))}{mk}B_m(s)\, ds\leq\sup_m\|B_m\|_{C([-1,0])}\frac{\cosh(mk)}{mk\sinh(mk)}.\]
Consequently, $u_m'(0)=w_m'(0)/(mk)^2\to_{m\to\infty}0.$ 
This proves our claim.
\end{proof}

We prove now our second main existence result, Theorem \ref{MT2}. 
Though it is independent of Theorem \ref{MT1}, the assertions of the latter are satisfied and we shall make use of some of the relations derived in its proof. 
\begin{proof}[Proof of Theorem \ref{MT2}]  Let $\sigma>0$ be fixed, and let $\Lambda $ and $\Lambda_-$ be the constants defined by \eqref{eq:bub} and \eqref{L-}, respectively.
Given $m\in\N$, $m\geq 1,$ and $\lambda\leq\Lambda_-$, it holds that  $w_m'(0)<0$ and, by \eqref{eq:PLd} and (A2),  
we get $\lim_{\lambda\to-\infty}\mu_m(\sigma,\lambda)=-\infty$ and $\mu_m(\sigma,\Lambda)=\sigma(mk)^2+g\ov\h(0)>0.$
Therefore,
\[
\ov\lambda_m(\sigma)=:\ov\lambda_m:={ \min}\{\lambda<\Lambda\,:\, \mu_m(\sigma,\lambda)=0\}
\]
is a well-defined constant for all $m\in\N.$  
 In order to apply Theorem \ref{CR} we have to make sure that $\mu_m(\sigma,\ov \lambda_n)\neq0$ for all $m\neq n $ and that $\p_\lambda\mu_m(\sigma,\ov\lambda_m)\neq0$ for all $m\geq 1.$

 We first claim that $\ov\lambda_m\to_{m\to\infty}-\infty.$ 
To prove this, we  assume that $(\ov\lambda_m)_m$ has a bounded subsequence (which we denote again by $(\ov\lambda_m)_m)$. 
From the arguments in Theorem~\ref{MT1}, where we established that $(\overline\sigma_m)$ converges to zero, we deduce, due to the boundedness of $(\ov\lambda_m)_m$, that $w'_m(0)/(mk)^2 \rightarrow 0$. Our claim follows by  dividing \eqref{eq:symb} by $(mk)^2$ and letting $m\to\infty$.
 
This shows  (also for $k=1$) that indeed $\ov\lambda_m\to_{m\to\infty}-\infty.$
 Whence, we may choose $K\in\N$ such that if $k\geq K,$ then   $\ov\lambda_m\leq\Lambda_- $ for all $m\geq1.$

We prove now that if $k\geq K$ and $m\geq1,$ then $\ov\lambda_m$ is the unique point within $(-\infty,\Lambda_-]$ such that $\mu_m(\sigma,\lambda)=0.$
To this end, we differentiate $\mu_m$ with respect to $\lambda$ and observe that, for $\lambda\leq\Lambda_-$,
 \[
\p_\lambda\mu_m(\sigma, \lambda)=-2\psi_{\lambda}'(0)\p_\lambda\psi_{\lambda}'(0)+\p_\lambda\psi_{\lambda}'(0)w_m'(0)+\psi_{\lambda}'(0)\p_\lambda w_m'(0)>0
\]
provided $\p_\lambda w_m'(0)\geq0.$ 
 We show now that this is indeed the case.
Differentiating both equations of \eqref{eq:BvD} with respect to $\lambda $, yields that $u:=\p_\lambda w_m$ is the solution of the following boundary value  problem
\begin{equation}\label{eq:BD}
\left\{
\begin{array}{rllll}
 u''-((mk)^2+\p_\psi f(y,\psi_\lambda(y)))u&=&\p_{\psi\psi}f(y,\psi_\lambda(y)))\p_\lambda\psi_\lambda w_m\\[1ex]
&&+\left(2\p_\psi f(y,\psi_\lambda(y))-g(1+y)\ov\h''(-\psi_\lambda)\right)\p_\lambda\psi_\lambda\\[1ex]
&&-(mk)^2(1+y)\p_\lambda \psi_\lambda'(y), \qquad -1<y<0,\\[1ex]
u(0)=u(-1)&=&0.
\end{array}
\right.
\end{equation}
It is not difficult to see that
\[
\p_\lambda\psi_\lambda'(y)=-1-\int_{-1}^0\int_{-1}^t\p_\psi f(s,\psi_\lambda(s))\p_\lambda\psi_\lambda(s)\, ds\, dt+\int_{-1}^y \p_\psi f(s,\psi_\lambda(s))\p_\lambda\psi_\lambda(s)\, ds,
\]
and gathering (A3),  (B1), (B3), $\psi_\lambda\leq 0,$ and $\p_\lambda\psi_\lambda\in[0,1]$ we conclude that 
\begin{align*}
\p_\lambda\psi_\lambda'(y)&\leq \p_\lambda\psi_\lambda'(0)=-1+\int_{-1}^0\int_{t}^0 \p_\psi f(s,\psi_\lambda(s))\p_\lambda\psi_\lambda(s)\, ds\, dt\\
& \leq-1+\int_{-1}^0\int_{t}^0 \p_\psi f(s,0)\, ds\, dt\leq 0.
\end{align*}
 Since $\lambda\leq\Lambda_-,$ we conclude from (B1) and (B2) that the  right-hand side of \eqref{eq:BD}$_1$ is positive.
Consequently,  $u\leq0$ and Hopf's principle implies that $u'(0)>0,$ that is $\p_\lambda w_m'(0)>0.$
This proves that $\ov\lambda_m$ is the unique solution of   $\mu_m(\sigma,\cdot)=0$   in $(-\infty,\Lambda_-]$, and shows moreover that
$\p_\lambda\mu_m(\sigma,\ov\lambda_m)>0$ for all $m\geq1.$

Finally, we prove that if $k\geq K,$ then $\ov\lambda_n\neq\ov\lambda_m$ for all $m\neq n.$
Let thus  $m>n\geq 1$.  
We assume by contradiction that $\ov\lambda_n=\ov\lambda_m=:\lambda$.
Then, since $0=\mu_m(\sigma,\lambda)=\mu_n(\sigma,\lambda),$
we conclude  
\[
(mk)^2\left(\sigma+\frac{\psi_\lambda'(0)w_m'(0)}{(mk)^2}\right)=(nk)^2\left(\sigma+\frac{\psi_\lambda'(0)w_n'(0)}{(nk)^2}\right)>0.
\] 
But,  for fixed $\lambda\leq\Lambda_-$, the sequence $(w_m'(0)/(mk)^2)$ is increasing and we have
\[\sigma+\frac{\psi_\lambda'(0)w_m'(0)}{(mk)^2}>\sigma+\frac{\psi_\lambda'(0)w_n'(0)}{(nk)^2}>0.\]
This shows that in fact $\mu_m(\sigma,\lambda)>\mu_n(\sigma,\lambda)$, which is a contradiction.
Therefore 
\begin{equation}\label{eq:B6}
\Kern \p_\eta\Psi(\sigma,\ov\lambda_m,0)=\spn\{\cos(mkx)\},\quad  \im \p_\eta\Psi(\sigma,\ov\lambda_m,0)\oplus \{\cos(mkx)\}=\wh C^{\alpha}_{e,k}(\mathbb S),
\end{equation} implying that $\p_\eta\Psi(\sigma,\ov\lambda_m,0)$ is a Fredholm operator of index zero.
Summarising, if $k\geq K,$   then $\mu_m(\sigma,\ov\lambda_n)=0$ if and only if $m=n,$ and $\p_\lambda\mu_m(\sigma,\ov\lambda_m)>0$ for all $m\geq 1.$
Thus,  we may apply Theorem \ref{CR} and obtain the desired result.
\end{proof}


\begin{acknowledgement*}
The authors would like to thank the organisers of the programme ``Nonlinear Water
Waves'',  Erwin Schr\"odinger Institute (Vienna), April--June 2011, where work on this paper was undertaken. DH would like to acknowledge the support of the Royal Irish Academy.
\end{acknowledgement*}

  \end{document}